\documentclass{article}

\newif\ifUseEepic
\UseEepicfalse

\ifUseEepic
  \usepackage{eepic} 
  \usepackage{color}
  \usepackage{rotate}
  \pdfoutput=0
\else
  \usepackage{tikz}
\fi

\usepackage{ifthen}
\usepackage{amsfonts}
\usepackage{amssymb}
\usepackage{amsmath}
\usepackage{amsthm}

\newcommand{\flipr}[1]{\mbox{\rotate[f]{\rotate[l]{$#1$}}}}
\newcommand{\flipv}[1]{\mbox{\rotate[f]{$#1$}}}
\newcommand{\fliph}[1]{\mbox{\rotate[f]{\rotate[u]{$#1$}}}}

\newcommand{\rotu}[1]{\mbox{\mbox{\rotate[u]{$#1$}}}}
\newcommand{\rotz}[1]{\mbox{\mbox{$#1$}}}

\newcommand{\C}{{\mathbb C}}

\newcommand{\N}{{\mathbb N}}
\newcommand{\R}{{\mathbb R}}

\newcommand{\I}{{\mathrm i}}

\newcommand{\jac}{{\mathcal J}}
\newcommand{\hess}{{\mathcal H}}

\newcommand{\diag}{\operatorname{diag}}

\newcommand{\grad}{\nabla}
\newcommand{\Mat}{\operatorname{Mat}}
\newcommand{\rk}{\operatorname{rk}}
\newcommand{\tr}{\operatorname{tr}}
\newcommand{\GL}{\operatorname{GL}}

\newcommand{\tp}{^{\rm t}}
\newcommand{\rv}{^{\rm r}}

\newcommand{\Rea}{\operatorname{Re}}
\newcommand{\Ima}{\operatorname{Im}}

\newcommand{\RightArroW}{{\mathversion{bold}$\Rightarrow$}}
\newcommand{\LeftArroW}{{\mathversion{bold}$\Leftarrow$}}

\newcommand{\antisym}{\discretionary{(anti-)}{}{(anti)}sym\-me\-try}
\newcommand{\antisyms}{\discretionary{(anti-)}{}{(anti)}sym\-me\-tries}
\newcommand{\parder}[3][Default]{
	\frac{\partial \ifthenelse{\equal{#1}{Default}}{}{^{#1}}#2}{
              \partial #3 \ifthenelse{\equal{#1}{Default}}{}{^{#1}}}}

\ifUseEepic
\newcommand{\bluefill}{\texture{cccccccc 0 cccccccc 0 cccccccc 0 cccccccc 0
                                cccccccc 0 cccccccc 0 cccccccc 0 cccccccc 0
                                cccccccc 0 cccccccc 0 cccccccc 0 cccccccc 0
                                cccccccc 0 cccccccc 0 cccccccc 0 cccccccc 0}}
\newcommand{\redfill}{\texture{44444444 0 0 0 44444444 0 0 0
                               44444444 0 0 0 44444444 0 0 0
                               44444444 0 0 0 44444444 0 0 0
                               44444444 0 0 0 44444444 0 0 0}}
\fi

\newcommand{\jc}{\ifUseEepic
\begin{picture}(14,2)(-2,2)
  \redfill 
  \shade\path(0,0)(10,0)(10,10)(0,10)(0,0)
\end{picture}
\else
\begin{tikzpicture}[x=1pt,y=1pt]
\useasboundingbox (-7,-3) rectangle (7,3);
\path[draw=black,fill=red!50] (-5,-5) rectangle (5,5);
\end{tikzpicture}
\fi
}
\newcommand{\jcs}{\ifUseEepic
\begin{picture}(14,2)(-2,2)
  \redfill 
  \shade\path(0,0)(10,0)(10,10)(0,10)(0,0)
  \put(5,5){\makebox(0,0){?}}
\end{picture}
\else
\begin{tikzpicture}[x=1pt,y=1pt]
\useasboundingbox (-7,-3) rectangle (7,3);
\path[draw=black,fill=red!50] (-5,-5) rectangle (5,5);
\path (0,0) node {?};
\end{tikzpicture}
\fi
}
\newcommand{\sjc}{\ifUseEepic
\begin{picture}(14,2)(-2,2)
  \redfill 
  \shade\path(0,10)(0,0)(10,0)
  \shade\path(10,0)(10,10)(0,10)
  \path(0,10)(10,0)
  \path(0,0)(10,0)(10,10)(0,10)(0,0)
\end{picture}
\else
\begin{tikzpicture}[x=1pt,y=1pt]
\useasboundingbox (-7,-3) rectangle (7,3);
\path[draw=black,fill=red!50] (-5,5) -- (-5,-5) -- (5,-5) -- cycle;
\path[draw=black,fill=red!50] (5,-5) -- (5,5) -- (-5,5) -- cycle;
\end{tikzpicture}
\fi
}
\newcommand{\asjc}{\ifUseEepic
\begin{picture}(14,2)(-2,2)
  \bluefill 
  \shade\path(10,0)(10,10)(0,10)
  \redfill
  \shade\path(0,10)(0,0)(10,0)
  \path(0,10)(10,0)
  \path(0,0)(10,0)(10,10)(0,10)(0,0)
\end{picture}
\else
\begin{tikzpicture}[x=1pt,y=1pt]
\useasboundingbox (-7,-3) rectangle (7,3);
\path[draw=black,fill=red!50] (-5,5) -- (-5,-5) -- (5,-5) -- cycle;
\path[draw=black,fill=blue!50] (5,-5) -- (5,5) -- (-5,5) -- cycle;
\end{tikzpicture}
\fi
}
\newcommand{\trasjc}{\ifUseEepic
\begin{picture}(14,2)(-2,2)
  \bluefill 
  \shade\path(0,0)(10,0)(10,10)
  \redfill
  \shade\path(5,10)(0,10)(0,0)(5,0)
  \shade\path(10,5)(10,10)(0,10)(0,5)
  \path(0,0)(10,10)
  \path(0,0)(10,0)(10,10)(0,10)(0,0)
\end{picture}
\else
\begin{tikzpicture}[x=1pt,y=1pt]
\useasboundingbox (-7,-3) rectangle (7,3);
\path[draw=black,fill=blue!50] (5,5) -- (5,-5) -- (-5,-5) -- cycle;
\path[draw=black,fill=red!50] (-5,-5) -- (-5,5) -- (5,5) -- cycle;
\path[draw=black,fill=red!50] (0,0) -- (-5,-5) -- (0,-5);
\path[draw=black,fill=red!50] (5,0) -- (5,5) -- (0,0);
\end{tikzpicture}
\fi
}
\newcommand{\dsjc}{\ifUseEepic
\begin{picture}(14,2)(-2,2)
  \bluefill 
  \shade\path(10,0)(10,10)(0,10)
  \redfill
  \shade\path(0,5)(0,0)(10,0)(10,5)
  \shade\path(5,10)(0,10)(0,0)(5,0)
  \path(10,0)(0,10)
\end{picture}
\else
\begin{tikzpicture}[x=1pt,y=1pt]
\useasboundingbox (-7,-3) rectangle (7,3);
\path[draw=black,fill=blue!50] (5,-5) -- (5,5) -- (-5,5) -- cycle;
\path[draw=black,fill=red!50] (-5,5) -- (-5,-5) -- (5,-5) -- cycle;
\path[draw=black,fill=red!50] (0,0) -- (-5,5) -- (0,5);
\path[draw=black,fill=red!50] (5,0) -- (5,-5) -- (0,0);
\end{tikzpicture}
\fi
}
\newcommand{\crjc}{\ifUseEepic
\begin{picture}(14,2)(-2,2)
  \redfill 
  \shade\path(10,10)(0,10)(5,5)(10,10)
  \shade\path(10,0)(10,10)(5,5)(10,0)
  \shade\path(0,0)(10,0)(5,5)(0,0)
  \shade\path(0,10)(0,0)(5,5)(0,10)
\end{picture}
\else
\begin{tikzpicture}[x=1pt,y=1pt]
\useasboundingbox (-7,-3) rectangle (7,3);
\path[draw=black,fill=red!50] (-5,-5) -- (0,0) -- (5,-5) -- cycle;
\path[draw=black,fill=red!50] (5,-5) -- (0,0) -- (5,5) -- cycle;
\path[draw=black,fill=red!50] (5,5) -- (0,0) -- (-5,5) -- cycle;
\path[draw=black,fill=red!50] (-5,5) -- (0,0) -- (-5,-5) -- cycle;
\end{tikzpicture}
\fi
}
\newcommand{\crudjc}{\ifUseEepic
\begin{picture}(14,2)(-2,2)
  \bluefill 
  \shade\path(10,10)(0,10)(5,5)(10,10)
  \redfill 
  \shade\path(10,0)(10,10)(5,5)(10,0)
  \bluefill 
  \shade\path(0,0)(10,0)(5,5)(0,0)
  \redfill 
  \shade\path(0,10)(0,0)(5,5)(0,10)
\end{picture}
\else
\begin{tikzpicture}[x=1pt,y=1pt]
\useasboundingbox (-7,-3) rectangle (7,3);
\path[draw=black,fill=blue!50] (-5,-5) -- (0,0) -- (5,-5) -- cycle;
\path[draw=black,fill=red!50] (5,-5) -- (0,0) -- (5,5) -- cycle;
\path[draw=black,fill=blue!50] (5,5) -- (0,0) -- (-5,5) -- cycle;
\path[draw=black,fill=red!50] (-5,5) -- (0,0) -- (-5,-5) -- cycle;
\end{tikzpicture}
\fi
}
\newcommand{\crurjc}{\ifUseEepic
\begin{picture}(14,2)(-2,2)
  \bluefill 
  \shade\path(10,10)(0,10)(5,5)(10,10)
  \shade\path(10,0)(10,10)(5,5)(10,0)
  \redfill 
  \shade\path(0,0)(10,0)(5,5)(0,0)
  \shade\path(0,10)(0,0)(5,5)(0,10)
\end{picture}
\else
\begin{tikzpicture}[x=1pt,y=1pt]
\useasboundingbox (-7,-3) rectangle (7,3);
\path[draw=black,fill=red!50] (-5,-5) -- (0,0) -- (5,-5) -- cycle;
\path[draw=black,fill=blue!50] (5,-5) -- (0,0) -- (5,5) -- cycle;
\path[draw=black,fill=blue!50] (5,5) -- (0,0) -- (-5,5) -- cycle;
\path[draw=black,fill=red!50] (-5,5) -- (0,0) -- (-5,-5) -- cycle;
\end{tikzpicture}
\fi
}
\newcommand{\crrdjc}{\ifUseEepic
\begin{picture}(14,2)(-2,2)
  \redfill 
  \shade\path(10,10)(0,10)(5,5)(10,10)
  \bluefill
  \shade\path(10,0)(10,10)(5,5)(10,0)
  \shade\path(0,0)(10,0)(5,5)(0,0)
  \redfill 
  \shade\path(0,10)(0,0)(5,5)(0,10)
\end{picture}
\else
\begin{tikzpicture}[x=1pt,y=1pt]
\useasboundingbox (-7,-3) rectangle (7,3);
\path[draw=black,fill=blue!50] (-5,-5) -- (0,0) -- (5,-5) -- cycle;
\path[draw=black,fill=blue!50] (5,-5) -- (0,0) -- (5,5) -- cycle;
\path[draw=black,fill=red!50] (5,5) -- (0,0) -- (-5,5) -- cycle;
\path[draw=black,fill=red!50] (-5,5) -- (0,0) -- (-5,-5) -- cycle;
\end{tikzpicture}
\fi
}
\newcommand{\crujc}{\ifUseEepic
\begin{picture}(14,2)(-2,2)
  \bluefill 
  \shade\path(10,10)(0,10)(5,5)(10,10)
  \redfill 
  \shade\path(10,0)(10,10)(5,5)(10,0)
  \shade\path(0,0)(10,0)(5,5)(0,0)
  \shade\path(0,10)(0,0)(5,5)(0,10)
\end{picture}
\else
\begin{tikzpicture}[x=1pt,y=1pt]
\useasboundingbox (-7,-3) rectangle (7,3);
\path[draw=black,fill=red!50] (-5,-5) -- (0,0) -- (5,-5) -- cycle;
\path[draw=black,fill=red!50] (5,-5) -- (0,0) -- (5,5) -- cycle;
\path[draw=black,fill=blue!50] (5,5) -- (0,0) -- (-5,5) -- cycle;
\path[draw=black,fill=red!50] (-5,5) -- (0,0) -- (-5,-5) -- cycle;
\end{tikzpicture}
\fi
}
\newcommand{\crrjc}{\ifUseEepic
\begin{picture}(14,2)(-2,2)
  \redfill
  \shade\path(10,10)(0,10)(5,5)(10,10)
  \bluefill 
  \shade\path(10,0)(10,10)(5,5)(10,0)
  \redfill 
  \shade\path(0,0)(10,0)(5,5)(0,0)
  \shade\path(0,10)(0,0)(5,5)(0,10)
\end{picture}
\else
\begin{tikzpicture}[x=1pt,y=1pt]
\useasboundingbox (-7,-3) rectangle (7,3);
\path[draw=black,fill=red!50] (-5,-5) -- (0,0) -- (5,-5) -- cycle;
\path[draw=black,fill=blue!50] (5,-5) -- (0,0) -- (5,5) -- cycle;
\path[draw=black,fill=red!50] (5,5) -- (0,0) -- (-5,5) -- cycle;
\path[draw=black,fill=red!50] (-5,5) -- (0,0) -- (-5,-5) -- cycle;
\end{tikzpicture}
\fi
}
\newcommand{\crdjc}{\ifUseEepic
\begin{picture}(14,2)(-2,2)
  \redfill 
  \shade\path(10,10)(0,10)(5,5)(10,10)
  \shade\path(10,0)(10,10)(5,5)(10,0)
  \bluefill
  \shade\path(0,0)(10,0)(5,5)(0,0)
  \redfill 
  \shade\path(0,10)(0,0)(5,5)(0,10)
\end{picture}
\else
\begin{tikzpicture}[x=1pt,y=1pt]
\useasboundingbox (-7,-3) rectangle (7,3);
\path[draw=black,fill=blue!50] (-5,-5) -- (0,0) -- (5,-5) -- cycle;
\path[draw=black,fill=red!50] (5,-5) -- (0,0) -- (5,5) -- cycle;
\path[draw=black,fill=red!50] (5,5) -- (0,0) -- (-5,5) -- cycle;
\path[draw=black,fill=red!50] (-5,5) -- (0,0) -- (-5,-5) -- cycle;
\end{tikzpicture}
\fi
}
\newcommand{\crljc}{\ifUseEepic
\begin{picture}(14,2)(-2,2)
  \redfill 
  \shade\path(10,10)(0,10)(5,5)(10,10)
  \shade\path(10,0)(10,10)(5,5)(10,0)
  \shade\path(0,0)(10,0)(5,5)(0,0)
  \bluefill 
  \shade\path(0,10)(0,0)(5,5)(0,10)
\end{picture}
\else
\begin{tikzpicture}[x=1pt,y=1pt]
\useasboundingbox (-7,-3) rectangle (7,3);
\path[draw=black,fill=red!50] (-5,-5) -- (0,0) -- (5,-5) -- cycle;
\path[draw=black,fill=red!50] (5,-5) -- (0,0) -- (5,5) -- cycle;
\path[draw=black,fill=red!50] (5,5) -- (0,0) -- (-5,5) -- cycle;
\path[draw=black,fill=blue!50] (-5,5) -- (0,0) -- (-5,-5) -- cycle;
\end{tikzpicture}
\fi
}
\newcommand{\rsjc}{\ifUseEepic
\begin{picture}(14,2)(-2,2)
  \redfill 
  \shade\path(0,0)(10,0)(10,10)
  \shade\path(10,10)(0,10)(0,0)
  \path(0,0)(10,10)
  \path(0,0)(10,0)(10,10)(0,10)(0,0)
\end{picture}
\else
\begin{tikzpicture}[x=1pt,y=1pt]
\useasboundingbox (-7,-3) rectangle (7,3);
\path[draw=black,fill=red!50] (5,5) -- (-5,5) -- (-5,-5) -- cycle;
\path[draw=black,fill=red!50] (-5,-5) -- (5,-5) -- (5,5) -- cycle;
\end{tikzpicture}
\fi
}
\newcommand{\rasjc}{\ifUseEepic
\begin{picture}(14,2)(-2,2)
  \bluefill
  \shade\path(0,0)(10,0)(10,10)
  \redfill 
  \shade\path(10,10)(0,10)(0,0)
  \path(0,0)(10,10)
  \path(0,0)(10,0)(10,10)(0,10)(0,0)
\end{picture}
\else
\begin{tikzpicture}[x=1pt,y=1pt]
\useasboundingbox (-7,-3) rectangle (7,3);
\path[draw=black,fill=red!50] (5,5) -- (-5,5) -- (-5,-5) -- cycle;
\path[draw=black,fill=blue!50] (-5,-5) -- (5,-5) -- (5,5) -- cycle;
\end{tikzpicture}
\fi
}
\newcommand{\rsnjc}{\ifUseEepic
\begin{picture}(14,2)(-2,2)
  \path(10,0)(10,10)(0,10)
  \redfill
  \shade\path(0,5)(0,0)(10,0)(10,5)
  \shade\path(5,10)(0,10)(0,0)(5,0)
  \path(0,0)(10,10)
\end{picture}
\else
\begin{tikzpicture}[x=1pt,y=1pt]
\useasboundingbox (-7,-3) rectangle (7,3);
\path[draw=black,fill=red!50] (5,5) -- (-5,5) -- (-5,-5) -- cycle;
\path[draw=black,fill=red!50] (-5,-5) -- (5,-5) -- (5,5) -- cycle;
\path[draw=black,fill=white] (5,0) -- (5,5) -- (0,0);
\path[draw=black,fill=white] (0,0) -- (5,5) -- (0,5);
\end{tikzpicture}
\fi
}
\newcommand{\cjc}{\ifUseEepic
\begin{picture}(14,2)(-2,2)
  \bluefill
  \shade\path(10,0)(10,10)(0,10)
  \redfill
  \shade\path(5,10)(0,10)(0,0)(5,0)
  \shade\path(0,5)(0,0)(10,0)(10,5)
  \put(5,5){\circle*{1}}
\end{picture}
\else
\begin{tikzpicture}[x=1pt,y=1pt]
\useasboundingbox (-7,-3) rectangle (7,3);
\path[fill=red!50] (-5,-5) rectangle (5,5);
\path[fill=blue!50] (0,0) rectangle (5,5);
\path[fill=black] (0,0) circle (1pt);
\draw[black] (-5,-5) rectangle (5,5);
\end{tikzpicture}
\fi
}
\newcommand{\acjc}{\ifUseEepic
\begin{picture}(14,2)(-2,2)
  \bluefill
  \shade\path(0,0)(10,0)(10,10)
  \redfill
  \shade\path(5,10)(0,10)(0,0)(5,0)
  \shade\path(0,5)(0,10)(10,10)(10,5)
  \put(5,5){\circle*{1}}
\end{picture}
\else
\begin{tikzpicture}[x=1pt,y=1pt]
\useasboundingbox (-7,-3) rectangle (7,3);
\path[fill=red!50] (-5,-5) rectangle (5,5);
\path[fill=blue!50] (0,0) rectangle (5,-5);
\path[fill=black] (0,0) circle (1pt);
\draw[black] (-5,-5) rectangle (5,5);
\end{tikzpicture}
\fi
}
\newcommand{\djc}{\ifUseEepic
\begin{picture}(14,2)(-2,2)
  \redfill
  \shade\path(0,0)(10,0)(10,10)(0,10)(0,0)
  \put(5,5){\circle*{1}}
\end{picture}
\else
\begin{tikzpicture}[x=1pt,y=1pt]
\useasboundingbox (-7,-3) rectangle (7,3);
\path[fill=red!50] (-5,-5) rectangle (5,5);
\path[fill=black] (0,0) circle (1pt);
\draw[black] (-5,-5) rectangle (5,5);
\end{tikzpicture}
\fi
}
\newcommand{\adjc}{\ifUseEepic
\begin{picture}(14,2)(-2,2)
  \bluefill
  \shade\path(0,0)(10,0)(10,10)(0,10)(0,0)
  \redfill
  \shade\path(5,10)(0,10)(0,0)(5,0)
  \put(5,5){\circle*{1}}
\end{picture}
\else
\begin{tikzpicture}[x=1pt,y=1pt]
\useasboundingbox (-7,-3) rectangle (7,3);
\path[fill=red!50] (-5,-5) rectangle (5,5);
\path[fill=blue!50] (0,-5) rectangle (5,5);
\path[fill=black] (0,0) circle (1pt);
\draw[black] (-5,-5) rectangle (5,5);
\end{tikzpicture}
\fi
}
\newcommand{\hvjc}{\ifUseEepic
\begin{picture}(14,2)(-2,2)
  \redfill
  \shade\path(0,0)(10,0)(10,10)(0,10)(0,0)
  \path(5,0)(5,10)
  \path(10,5)(0,5)
\end{picture}
\else
\begin{tikzpicture}[x=1pt,y=1pt]
\useasboundingbox (-7,-3) rectangle (7,3);
\path[draw=black,fill=red!50] (-5,0) rectangle (0,-5);
\path[draw=black,fill=red!50] (0,-5) rectangle (5,0);
\path[draw=black,fill=red!50] (5,0) rectangle (0,5);
\path[draw=black,fill=red!50] (0,5) rectangle (-5,0);
\end{tikzpicture}
\fi
}
\newcommand{\ahvjc}{\ifUseEepic
\begin{picture}(14,2)(-2,2)
  \bluefill
  \shade\path(0,0)(10,0)(10,5)(0,5)(0,0)
  \redfill
  \shade\path(0,5)(10,5)(10,10)(0,10)(0,5)
  \path(5,0)(5,10)
\end{picture}
\else
\begin{tikzpicture}[x=1pt,y=1pt]
\useasboundingbox (-7,-3) rectangle (7,3);
\path[draw=black,fill=blue!50] (-5,0) rectangle (0,-5);
\path[draw=black,fill=blue!50] (0,-5) rectangle (5,0);
\path[draw=black,fill=red!50] (5,0) rectangle (0,5);
\path[draw=black,fill=red!50] (0,5) rectangle (-5,0);
\end{tikzpicture}
\fi
}
\newcommand{\havjc}{\ifUseEepic
\begin{picture}(14,2)(-2,2)
  \bluefill
  \shade\path(5,0)(10,0)(10,10)(5,10)(5,0)
  \redfill
  \shade\path(0,0)(5,0)(5,10)(0,10)(0,0)
  \path(10,5)(0,5)
\end{picture}
\else
\begin{tikzpicture}[x=1pt,y=1pt]
\useasboundingbox (-7,-3) rectangle (7,3);
\path[draw=black,fill=red!50] (-5,0) rectangle (0,-5);
\path[draw=black,fill=blue!50] (0,-5) rectangle (5,0);
\path[draw=black,fill=blue!50] (5,0) rectangle (0,5);
\path[draw=black,fill=red!50] (0,5) rectangle (-5,0);
\end{tikzpicture}
\fi
}
\newcommand{\ahavjc}{\ifUseEepic
\begin{picture}(14,2)(-2,2)
  \bluefill
  \shade\path(5,5)(10,5)(10,10)(5,10)(5,5)
  \shade\path(0,0)(5,0)(5,5)(0,5)(0,0)
  \redfill
  \shade\path(0,5)(5,5)(5,10)(0,10)(0,5)
  \shade\path(5,0)(10,0)(10,5)(5,5)(5,0)
\end{picture}
\else
\begin{tikzpicture}[x=1pt,y=1pt]
\useasboundingbox (-7,-3) rectangle (7,3);
\path[draw=black,fill=blue!50] (-5,0) rectangle (0,-5);
\path[draw=black,fill=red!50] (0,-5) rectangle (5,0);
\path[draw=black,fill=blue!50] (5,0) rectangle (0,5);
\path[draw=black,fill=red!50] (0,5) rectangle (-5,0);
\end{tikzpicture}
\fi
}
\newcommand{\hvsjc}{\ifUseEepic
\begin{picture}(14,2)(-2,2)
  \redfill
  \shade\path(0,0)(10,0)(10,10)(0,10)(0,0)
  \path(5,0)(5,10)
  \path(10,5)(0,5)
  \path(0,0)(10,10)
  \path(0,10)(10,0)
\end{picture}
\else
\begin{tikzpicture}[x=1pt,y=1pt]
\useasboundingbox (-7,-3) rectangle (7,3);
\path[draw=black,fill=red!50] (-5,0) rectangle (0,-5);
\path[draw=black,fill=red!50] (0,-5) rectangle (5,0);
\path[draw=black,fill=red!50] (5,0) rectangle (0,5);
\path[draw=black,fill=red!50] (0,5) rectangle (-5,0);
\draw[black] (-5,-5) -- (5,5);
\draw[black] (5,-5) -- (-5,5);
\end{tikzpicture}
\fi
}
\newcommand{\ahvsjc}{\ifUseEepic
\begin{picture}(14,2)(-2,2)
  \bluefill
  \shade\path(0,0)(10,0)(10,5)(0,5)(0,0)
  \redfill
  \shade\path(0,5)(10,5)(10,10)(0,10)(0,5)
  \path(5,0)(5,10)
  \path(0,0)(10,10)
  \path(0,10)(10,0)
\end{picture}
\else
\begin{tikzpicture}[x=1pt,y=1pt]
\useasboundingbox (-7,-3) rectangle (7,3);
\path[draw=black,fill=blue!50] (-5,0) rectangle (0,-5);
\path[draw=black,fill=blue!50] (0,-5) rectangle (5,0);
\path[draw=black,fill=red!50] (5,0) rectangle (0,5);
\path[draw=black,fill=red!50] (0,5) rectangle (-5,0);
\draw[black] (-5,-5) -- (5,5);
\draw[black] (5,-5) -- (-5,5);
\end{tikzpicture}
\fi
}
\newcommand{\havsjc}{\ifUseEepic
\begin{picture}(14,2)(-2,2)
  \bluefill
  \shade\path(5,0)(10,0)(10,10)(5,10)(5,0)
  \redfill
  \shade\path(0,0)(5,0)(5,10)(0,10)(0,0)
  \path(10,5)(0,5)
  \path(0,0)(10,10)
  \path(0,10)(10,0)
\end{picture}
\else
\begin{tikzpicture}[x=1pt,y=1pt]
\useasboundingbox (-7,-3) rectangle (7,3);
\path[draw=black,fill=red!50] (-5,0) rectangle (0,-5);
\path[draw=black,fill=blue!50] (0,-5) rectangle (5,0);
\path[draw=black,fill=blue!50] (5,0) rectangle (0,5);
\path[draw=black,fill=red!50] (0,5) rectangle (-5,0);
\draw[black] (-5,-5) -- (5,5);
\draw[black] (5,-5) -- (-5,5);
\end{tikzpicture}
\fi
}
\newcommand{\ahavsjc}{\ifUseEepic
\begin{picture}(14,2)(-2,2)
  \bluefill
  \shade\path(5,5)(10,5)(10,10)(5,10)(5,5)
  \shade\path(0,0)(5,0)(5,5)(0,5)(0,0)
  \redfill
  \shade\path(0,5)(5,5)(5,10)(0,10)(0,5)
  \shade\path(5,0)(10,0)(10,5)(5,5)(5,0)
  \path(0,0)(10,10)
  \path(0,10)(10,0)
\end{picture}
\else
\begin{tikzpicture}[x=1pt,y=1pt]
\useasboundingbox (-7,-3) rectangle (7,3);
\path[draw=black,fill=blue!50] (-5,0) rectangle (0,-5);
\path[draw=black,fill=red!50] (0,-5) rectangle (5,0);
\path[draw=black,fill=blue!50] (5,0) rectangle (0,5);
\path[draw=black,fill=red!50] (0,5) rectangle (-5,0);
\draw[black] (-5,-5) -- (5,5);
\draw[black] (5,-5) -- (-5,5);
\end{tikzpicture}
\fi
}
\newcommand{\axjc}{\ifUseEepic
\begin{picture}(14,2)(-2,2)
  \redfill
  \shade\path(0,0)(10,0)(10,10)(0,10)(0,0)
  \whiten\path(0,10)(4.6,10)(4.6,6)(0,6)(0,10)
  \whiten\path(10,0)(6,0)(6,4.6)(10,4.6)(10,0)
  \thicklines\color{white}\path(4.6,10)(4.6,6)(0,6)
  \thicklines\color{white}\path(6,0)(6,4.6)(10,4.6)
  \thinlines\color{black}\path(0,0)(10,0)(10,10)(0,10)(0,0)
\end{picture}
\else
\begin{tikzpicture}[x=1pt,y=1pt]
\useasboundingbox (-7,-3) rectangle (7,3);
\path[fill=red!50] (-5,-5) rectangle (5,5);
\path[fill=white] (0,-5) rectangle (5,0);
\path[fill=white] (-5,0) rectangle (0,5);
\draw[black] (-5,-5) rectangle (5,5);
\end{tikzpicture}
\fi
}

\theoremstyle{plain}
\newtheorem{theorem}{Theorem}[section]

\newtheorem{lemma}[theorem]{Lemma}
\newtheorem{corollary}[theorem]{Corollary}
\theoremstyle{definition}
\newtheorem{definition}[theorem]{Definition}
\theoremstyle{remark}

\theoremstyle{plain}

\title{Symmetric Jacobians\footnote{An earlier version of this article appeared as Chapter 2 
in the author's Ph.D. thesis \cite{homokema}.}}

\author{Michiel de Bondt\footnote{The author's Ph.D. project was supported by The Netherlands
                                  Organization for Scientific Research (NWO).} \\
Department of Mathematics, Radboud University \\ 
Nijmegen, The Netherlands \\
\emph{E-mail:} M.deBondt@math.ru.nl}

\allowdisplaybreaks

\begin{document}

\maketitle

\begin{abstract}
\noindent
This article is about polynomial maps with a certain symmetry and/or antisymmetry
in their Jacobians, and whether
the Jacobian Conjecture is satisfied for such maps, or whether it
is sufficient to prove the Jacobian Conjecture for such maps.

For instance, we show that it suffices to prove the Jacobian conjecture
for polynomial maps $x + H$ over $\C$ such that $\jac H$ satisfies all 
symmetries of the square, where $H$ is homogeneous of arbitrary degree 
$d \ge 3$.
\end{abstract}



\section*{Introduction}

Let $F$ be a polynomial map over a field $K$ of characteristic zero.
The Jacobian Conjecture asserts that $F$ has a polynomial inverse in case
its Jacobian determinant $\det \jac F$ is a unit in $K$. It has been shown that in 
order to prove the Jacobian conjecture for all fields of characteristic zero, one can 
take an arbitrary such field $K$ and a favorite integer $d \ge 3$, after which it 
suffices to prove the Jacobian Conjecture for polynomial maps over $K$
of the form $F = x + H$, where $x$ is the identity map and $H$ is homogeneous 
of degree $d$.

We write $\jac F$ for the Jacobian of a polynomial map $F$, $\hess f$ for the 
Hessian of a single polynomial $f$, and $\grad f$ for the gradient map of
a single polynomial $f$. Notice that $\hess f = \jac \grad f$.
All results in this paper are about maps of the form $F = x + H$, with 
\antisym\ conditions on $\jac H$. 

Most results are about two such maps, 
say $F = x + H$ and $\tilde{F} = \tilde{x} + \tilde{H}$, each with their own 
\antisym\ conditions on $\jac H = \jac_x H$ and $\jac_{\tilde{x}} \tilde{H}$ respectively, where the 
dimension of $\tilde{F}$ is one, two or four times that of $F$ (depending on the actual result), 
and $\tilde{x}$ is the identity in the proper dimension. For each of these results, 
the reader may choose any of the following additional conditions when desired:
\begin{itemize}

\item $\jac_x H$ and $\jac_{\tilde{x}} \tilde{H}$ are both singular or even nilpotent,

\item for some fixed arbitrary subset $S \subseteq \N$, $H$ and 
$\tilde{H}$ only have terms whose degrees are contained in $S$.

\end{itemize}
One can, e.g., assume that $H$ and $\tilde{H}$ only have terms of degree 
three, in which case the nilpotency of $\jac_x H$ and $\jac_{\tilde{x}} \tilde{H}$
already follows from the Keller condition. We have this Keller condition implicitly in case 
the actual result is the equivalence of the Jacobian conjecture for maps of the form $F = x + H$ 
and that for maps of the form $\tilde{F} = \tilde{x} + \tilde{H}$, since the Jacobian conjecture 
is an assertion about Keller maps.

We associate vectors with column matrices and write $M\tp$ for the {\em transpose}
of a matrix $M$. We write $M\rv$ for the {\em reverse} of a matrix $M$, i.e.\@
if $M$ has $n$ rows, then the $i$-th row of $M\rv$ is equal to the $(n+1-i)$-th
row of $M$ for each $i$. Notice that the symmetries corresponding to the matrix operators 
$M \mapsto M\tp$ and $M \mapsto M\rv$ generate the whole
dihedral symmetry group of the square. If $f$ is a single polynomial, then we can
view $f$ as a polynomial map with only one component, and we have
$\jac f = (\grad f)\tp$.

Both $x = x_1, x_2, \ldots, x_n$ and $y = y_1, y_2, \ldots, 
y_n$ are $n$-tuples of variables, thus $\C[x,y]$ is the coordinate ring of
complex $2n$-space. By taking into account the order of variables in $\C[x,y]$,
$(x,y)$ is the identity map of the above complex space. But $(x,y)$ is also a 
vector of $2n$ variables.

Notice that besides the matrix equality 
\begin{equation}
M\rv M' = (I_m\rv M) M' = I_m\rv (M M') = (MM')\rv \label{mrev}
\end{equation}
where $M$ has height $m$ and $I_m$ is the identity matrix of size $m$,
we have the following equalities:
\begin{alignat}{2}
\hess_X\rv f &:= (\hess_X f)\rv &&= \big(\jac_X (\grad_X f)\big)\rv \label{hrev} \\
\jac_X\rv  F &:= (\jac_X F)\rv  &&= \jac_X (F\rv) \label{jrev} \\
\grad_X\rv f &:= (\grad_X f)\rv &&= \grad_{X\rv} f \label{grev}
\end{alignat}

We will use the above equalities in the rest of this paper, which is organized as follows.
In section \ref{sec1} and in theorem \ref{sdjck}, we formulate results about the Jacobian conjecture for 
polynomial maps $x + H$ such that $\jac H$ has certain \antisym\ properties with respect to the
diagonal and/or the antidiagonal. In section \ref{sec2}, we formulate results about the 
Jacobian conjecture for polynomial maps $x + H$ such that $\jac H$ has certain \antisym\ properties 
with respect to the center, possibly among other \antisym\ properties. The reason that theorem \ref{sdjck}
contains results that belong to section \ref{sec1}, is that maps $x + H$, such 
that $\jac H$ has certain \antisym\ properties with respect to both the diagonal and the antidiagonal,
belong to section \ref{sec2} as well. 

In section \ref{sec3}, we formulate results about the (linear) dependence problem (for Jacobians)
for polynomial maps $x + H$ such that $\jac H$ has certain \antisym\ properties. For the definition
of this dependence problem, we refer to the beginning of section \ref{sec3}. At last, the conclusion 
follows.

\section{Diagonally symmetric variants of the Jacobian conjecture} \label{sec1}

We define \antisym\ properties by pictures that visualize them.

\begin{definition}
$\jc(K,n)$ means that the Jacobian conjecture is satisfied for 
$n$-dimensional maps $F = x + H$ over the field $K$, i.e., $F$ is invertible in case 
$\det \jac F \in K^{*}$, such that the degree of each term of $H$
is contained in a fixed set $S \subseteq \N$, and optionally $\det \jac H = 0$ 
or even $(\jac H)^n = 0$. 

So we do not assume that $\det JF = 1$ necessarily, except when
we assume $(\jac H)^n = 0$, in which case $\det JF = 1$ as a consequence, or
$1 \notin S$, in which case $\det JF = 1$ follows from $\det \jac F \in K^{*}$.

$\sjc(K,n)$ and $\rsjc(K,n)$ mean that the Jacobian conjecture is
satisfied for $n$-dimensional maps $F = x + H$ over $K$, which 
have a symmetric Jacobian with respect to the diagonal and the 
anti-diagonal respectively, where $H$ has the same partially 
chosen properties as in the definition of $\jc(K,n)$.

$\asjc(K,n)$ and $\rasjc(K,n)$ mean that the Jacobian conjecture is
satisfied for $n$-dimensional maps $F = x + H$ over $K$, for which $\jac H$ is 
anti-symmetric (i.e.\@ applying the `symmetry' negates the matrix)
with respect to the diagonal and the anti-diagonal respectively, where
$H$ has the same partially chosen properties as in the definition of $\jc(K,n)$.

In the definition of $\dsjc(K,n)$, the `symmetry' is partially an antisymmetry, 
namely where colors on opposite sides of the diagonal do not match.
\end{definition}

\noindent
In the proofs of our results, we will use $\jc$, $\sjc$, $\rsjc$
$\asjc$, $\rasjc$ and $\dsjc$ without parenthesized arguments to indicate the 
corresponding matrix \antisym. So we have, e.g., $\jac H = \jac H\tp$ if $\jac H$ has 
symmetry $\sjc$, and $\jac H = -\jac H\tp$ in case $\jac H$ has antisymmetry $\asjc$.

If the \antisym\ separates
several parts of the square, then entries on the edge of separation must
satisfy both \antisym\ conditions. In this manner, the row and column in the 
middle of matrices with \antisym\ $\dsjc$ must be zero when the dimension is odd,
since they are equal to each other and to the opposites of each other.

For matrices with antisymmetry $\asjc(K,n)$, the numbers on the diagonal must
be equal to their opposites by definition of $\asjc$, because the antisymmetry holds 
for the edge of the light and dark regions as well. So $\asjc$ is just the regular 
matrix antisymmetry with zeroes on the diagonal.

\begin{theorem}[Meng] \label{menghess}
Assume $K$ is a field of characteristic zero. 
Then $\rsjc(K,2n)$ implies $\jc(K,n)$. 
\end{theorem}

\begin{proof}
Assume $F = x + H \in K[x]^n$ such that $\jac H$ has symmetry $\jc$.
Put $f := y\tp F = \sum_{i=1}^n y_i F_i$. Then 
$\hess_{x,y} f(x,y\rv)$  has the regular (Hessian) symmetry $\sjc$, whence 
$$
\hess_{x,y}\rv f(x,y\rv) 
\stackrel{\eqref{hrev}}= \jac_{x,y}\rv \grad_{x,y} f (x,y\rv)
\stackrel{\eqref{jrev}}= \jac_{x,y} \grad_{x,y}\rv f (x,y\rv)
\stackrel{\eqref{grev}}= \jac_{x,y} \grad_{y\rv,x\rv} f(x,y\rv)
$$
has symmetry $\rsjc$. Since the first
$n$ components of $\grad_{y\rv,x\rv} f(x,y\rv)$ are equal to 
$\grad_{y\rv} (y\rv)\tp F = F$, we see that $\jac_{x,y} 
\grad_{y\rv,x\rv} f(x,y\rv)$ is of the form
$$
\jac_{x,y} \grad_{y\rv,x\rv} f(x,y\rv) = \left( \begin{array}{cc} 
    \jac F & {\mathbf 0} \\
    * & (((\jac F)\rv)\tp)\rv
\end{array} \right) = \left( \ifUseEepic \begin{array}{cc} 
    \jac F & {\mathbf 0} \\ 
    * & \flipr{\!\!\!\jac F~} 
\end {array} \else \begin{array}{cc}
    \jac F & {\mathbf 0} \\ 
    \begin{tikzpicture}[x=1pt,y=1pt] \useasboundingbox (-5,-5) rectangle (5,10);           
    \path (0,0) node{$*$}; \end{tikzpicture} &
    \begin{tikzpicture}[x=1pt,y=1pt] \useasboundingbox (-5,-5) rectangle (5,10);           
    \path (0,0) node[xscale=-1,rotate=90]{$\jac F$}; \end{tikzpicture}
\end{array} \fi \right)
$$
where the zero submatrix right above appears since $\deg_y f < 2$, and the part 
$y\tp x$ of $f$ does not affect $*$ because $\deg_x (y\tp x) < 2$. 
Thus $*$ only depends on $y\tp H$.

Hence $\grad_{y\rv,x\rv} f(x,y\rv)$ is of the form $\tilde{F} = (x,y) + \tilde{H}$ 
where $\jac_{x,y} \tilde{H}$ has symmetry $\rsjc$ and the properties
to be chosen by the reader of $\tilde{H}$ correspond to those of $H$. Thus if we assume 
$\rsjc(K,2n)$, then $\grad_{y\rv,x\rv} f(x,y\rv)$ satisfies the Jacobian Conjecture.
Since substituting $y = 0$ in $\grad_{y\rv,x\rv} f(x,y\rv)$ gives $(F,0)$, we see that 
$F$ satisfies the Jacobian conjecture as well. This gives the desired result.
\end{proof}

\noindent
In \cite{meng}, the author G. Meng constructs the map $f := \sum_{i=1}^n y_i F_i$
in the above proof. The corresponding gradient map $\grad_{x,y} f$
has symmetry $\sjc$, but its linear part is $(y,x)$
in case $F$ has linear part $x$. In order to restore the
linear part to $(x,y)$, we composed $\grad_{x,y} f$ with linear maps in the above 
proof, resulting $\grad_{y\rv,x\rv} f(x,y\rv)$ with linear part $(x,y)$ and 
symmetry $\rsjc$. That is why the above theorem is considered to be due to 
Meng. The case that $\jac H$ is nilpotent of corollary \ref{hessred} below was proved 
in \cite{art6}.

\begin{theorem} \label{hessequiv}
$\sjc(\C,N)$ and $\rsjc(\C,N)$ are equivalent.
\end{theorem}

\begin{proof}
Notice that it suffices to show that polynomial maps $H \in \C[X]^N$ with Jacobian symmetry
$\sjc$ can be transformed to polynomial maps $\tilde{H} \in \C[X]^N$ with Jacobian symmetry
$\rsjc$ by way of linear conjugation, and vice versa.
We shall show that this is the case, where the conjugation map has the symmetric
unitary Jacobian $T := \frac12 \sqrt{2} (I_N + \I I_N\rv)$.
\begin{description}

\item[(\RightArroW):] 
Let $F = X + H \in \C[X]^N$ such that $\jac_{X} H$ has symmetry $\sjc$, where 
$X = (x_1, x_2, \ldots, x_N)$. Since $T$ is symmetric, we see that $TH(TX)$ has the regular
Jacobian symmetry $\sjc$ as well. But $T^{-1} = \frac12 \sqrt{2} (I_N - \I I_N\rv) = -\I T\rv$, 
thus by \eqref{mrev}, $T^{-1}H(TX) = -\I(TH(TX))\rv$ has Jacobian symmetry $\rsjc$.
Since conjugations preserve the identity part $X$ of $F$, we see that $\tilde{F} := T^{-1}F(TX) = X + 
T^{-1}H(TX)$, so $\tilde{F} - X = T^{-1}H(TX)$ has Jacobian symmetry $\rsjc$.

\item[(\LeftArroW):] 
Let $F = X + H \in \C[X]^N$ such that $\jac_{X} H$ has symmetry $\rsjc$, where 
$X = (x_1, x_2, \ldots, x_N)$. Then $\jac_X H\rv$ has 
symmetry $\sjc$. Since $H\rv = I_N\rv H$ and $I_N\rv$ commutes with $T$, we obtain by \eqref{mrev}
and by symmetry of $T$ that
$$
I_N\rv T^{-1}H(TX) = T^{-1}H\rv(TX) = -\I T\rv H\rv (TX) = -\I(TH\rv(TX))\rv
$$ 
has Jacobian symmetry $\rsjc$. 
Thus $T^{-1}H(TX)$ has Jacobian symmetry $\sjc$ on account of \eqref{mrev}. 
Since conjugations preserve the identity part $X$ of $F$, we see that $\tilde{F} := T^{-1}F(TX) = X + 
T^{-1}H(TX)$, so $\tilde{F} - X = T^{-1}H(TX)$ has Jacobian symmetry $\sjc$. \qedhere
 
\end{description}
\end{proof}

\begin{corollary} \label{hessred}
$\sjc(\C,2n)$ implies $\jc(\C,n)$.
\end{corollary}

\begin{proof}
This follows from theorem \ref{hessequiv} and theorem \ref{menghess}.
\end{proof}

\begin{theorem}
Assume $K$ is a field of characteristic zero. Then $\rsjc(K,2n)$, $\dsjc(K,2n)$ 
and $\dsjc(K,2n+1)$ are equivalent. 
\end{theorem}

\begin{proof}
The equivalence of $\dsjc(K,2n)$ 
and $\dsjc(K,2n+1)$ follows by stabilization: both the row and column in the middle of
square matrices of odd dimension with \antisym\ $\dsjc$ are zero. Thus 
$\rsjc(K,2n) \Leftrightarrow \dsjc(K,2n)$ remains to be proved. 

Notice that it suffices to show that polynomial maps $H \in K[x,y]^{2n}$ with Jacobian symmetry
$\rsjc$ can be transformed to polynomial maps $\tilde{H} \in K[x,y]^{2n}$ with Jacobian symmetry
$\dsjc$ by way of linear conjugation, and vice versa.
We shall show that this is the case, where the conjugation map has the Jacobian
$$
T := \left( \begin{array}{cc} 
     I_n & I_n\rv \\ -I_n\rv & I_n \end{array} \right)
$$
Notice that $T\tp T = 2I_{2n}$ and $T\tp = I_{2n}\rv T I_{2n}\rv$. 
Thus $2T^{-1} =  T\tp = I_{2n}\rv T I_{2n}\rv$.
\begin{description}

\item[(\RightArroW):] 
Let $F = (x,y) + H \in K[x,y]^n$ such that $\jac_{x,y} H$ has symmetry $\rsjc$. 
Then $H\rv$ and hence also
$$
T\tp H\rv(T(x,y)) \stackrel{\eqref{mrev}}= 
T\tp I_{2n}\rv H(T(x,y)) = 2T^{-1} I_{2n}\rv H(T(x,y))
$$ 
has the regular Jacobian symmetry $\sjc$. Since negating the upper half of 
$2T^{-1} = T\tp$ has the same effect as reversing its columns, which is what
$I_{2n}\rv$ does in the product $2T^{-1} I_{2n}\rv$, we have that
$$
2T^{-1} I_{2n}\rv H(T(x,y)) =  2 J_{2n} T^{-1} H(T(x,y))
$$
where
$$
J_{2n} := \left( \begin{array}{cc} -I_n & {\mathbf 0} \\ {\mathbf 0} & I_n \end{array} \right)
$$
Combining the multiplication by $2J_{2n}$ 
with the regular Jacobian symmetry $\sjc$, we see that
$T^{-1} H(T(x,y))$ has Jacobian \antisym\ $\dsjc$. Since conjugations 
preserve the identity part $(x,y)$ of $F$, we see that $\tilde{F} := T^{-1}F(T(x,y)) = (x,y) + 
T^{-1}H(T(x,y))$, so $\tilde{F} - (x,y) = T^{-1}H(T(x,y))$ has Jacobian \antisym\ $\dsjc$.

\item[(\LeftArroW):] 
Let $F = (x,y) + H \in K[x,y]^n$ such that $\jac_{x,y} H$ has \antisym\ $\dsjc$. 
By negating the right half of $\jac_{x,y} H$, we see that $H(x,-y) = H(-J_{2n}(x,y))$ and hence also 
$M\tp H\big(-J_{2n} M(x,y)\big)$ has Jacobian symmetry $\sjc$ for each
$M \in \GL_{2n}(K)$.
In particular
$$
T\tp (-J_{2n}\tp) H\big( (-J_{2n})^2 T (x,y)\big) 
= -T\tp J_{2n} H\big( T (x,y)\big) 
$$ 
has Jacobian symmetry $\sjc$. Since negating the right half of 
$T\tp = 2 T^{-1}$ has the same effect as reversing the order of its rows, we have that 
$$
T\tp (-J_{2n}\tp) H\big( T (x,y)\big) 
= (2T^{-1})\rv H\big( T (x,y)\big) 
\stackrel{\eqref{mrev}}= \Big(2T^{-1} H\big(T(x,y)\big)\Big)\rv
$$
Thus $T^{-1} H(T(x,y))$ 
has Jacobian symmetry $\rsjc$. Since conjugations 
preserve the identity part $(x,y)$ of $F$, we see that $\tilde{F} := T^{-1}F(T(x,y)) = (x,y) + 
T^{-1}H(T(x,y))$, so $\tilde{F} - (x,y) = T^{-1}H(T(x,y))$ has Jacobian symmetry $\rsjc$. \qedhere
 
\end{description}
\end{proof}

\begin{corollary}[Dru{\.z}kowski] \label{druzhess}
Assume $K$ is a field of characteristic zero. Then $\dsjc(K,2n)$ implies $\jc(K,n)$. 
\end{corollary}

\begin{proof}
This follows immediately from the above theorem and
theorem \ref{menghess}.
\end{proof}

\noindent
In fact, Dru{\.z}kowski considers maps with \antisym\ $\sjc$,
but linear part $(-x,y) = J_{2n} (x,y)$ in \cite{druzsym}. Negating the first half of the map
restores the linear part, and the \antisym\ becomes $\dsjc$.

\subsection*{Symmetry patterns that satisfy the Jacobian Conjecture}

In some cases, the Jacobian Conjecture holds for polynomial maps $F = x + H$ with 
certain \antisyms\ of $\jac H$, because $F$ appears to be linear.

\begin{theorem}
Assume $F = x+H$ is a Keller map over $\R$, such that $\jac H$ has regular
symmetry $\sjc$. If either $H$ has no linear terms or $\jac H$ is nilpotent,
then $H$ is constant. 
In particular, $F$ is invertible because $F$ is translation in that case.
\end{theorem}

\begin{proof}
The case that $H$ has no linear terms follows from \cite[Cor.\@ 4.4]{dillenext}, so
assume that $(\jac H)^r = 0$ and $(\jac H)^{r-1} \ne 0$. If $r \ge 2$, then
$$
0=(\jac H)^{2r-2} = (\jac H)^{r-1}\cdot(\jac H)^{r-1} 
= (\jac H)^{r-1}\cdot\big((\jac H)^{r-1}\big)\tp 
$$
Substituting generic reals in the variables in the
rows of $(\jac H)^{r-1}$, we obtain rows of real numbers
that are isotropic (self-orthogonal), and hence zero. 
Contradiction, so $r=1$ and $\jac H=0$. Thus $F = x + H$ is a translation.
\end{proof}

\begin{theorem} \label{asymjc}
Assume $K$ is a field of characteristic zero. 
Then $\asjc(K,n)$ and $\rasjc(K,n)$ have affirmative answers.
In particular, if the Jacobian of a polynomial map $H$ over $K$ has antisymmetry 
$\asjc$ or $\rasjc$, then $\deg H \le 1$.
\end{theorem}

\begin{proof}
Assume that $H$ is a polynomial map over $K$ with Jacobian antisymmetry $\asjc$.
(The proof for Jacobian antisymmetry $\rasjc$ will be similar.) Then
$$
\parder{}{x_i}\parder{}{x_j} H_k = - \parder{}{x_i}\parder{}{x_k} H_j 
= \parder{}{x_j} \parder{}{x_k} H_i = -\parder{}{x_i}\parder{}{x_j} H_k
$$
and hence $2 \parder{}{x_i}\parder{}{x_j} H_k = 0$, for all $i,j,k$. So
$\deg H \le 1$. 
\end{proof}

\begin{definition}
$\crjc(K,n)$ means that the Jacobian conjecture is
satisfied for polynomial maps $F$ over $K$ that have a symmetric Jacobian 
with respect to both the diagonal and the anti-diagonal, where $H$ has the same
partially chosen properties as in the definition of $\jc(K,n)$.

In the definitions of $\crudjc(K,n)$, $\crurjc(K,n)$ and $\crrdjc(K,n)$,
some `symmetries' are antisymmetries, namely when colors on opposite sides
of the symmetry axis do not match.

In the definitions of $\crujc(K,n)$, $\crrjc(K,n)$, $\crdjc(K,n)$ 
and $\crljc(K,n)$, the `symmetries' are partially antisymmetries. 
\end{definition}

\noindent
Notice that $\crudjc(K,n)$, $\crurjc(K,n)$ and $\crrdjc(K,n)$ have affirmative
answers as well as $\asjc(K,n)$ and $\rasjc(K,n)$, because
the corresponding \antisyms\ are stronger than at least one of those of
$\asjc$ and $\rasjc$ in theorem \ref{asymjc}. 

\begin{theorem}
Assume $K$ is a field of characteristic zero. Then $\crujc(K,n)$, 
$\crrjc(K,n)$, $\crdjc(K,n)$ and $\crljc(K,n)$ have affirmative answers.
In particular, if the Jacobian of a polynomial map $H$ over $K$ has \antisym\ 
$\crujc$, $\crrjc$, $\crdjc$ or $\crljc$, then $\deg H \le 1$.
\end{theorem}

\begin{proof}
Assume that $H$ is a polynomial map over $K$ with Jacobian \antisym\ 
$\crujc$ or $\crljc$. (The proof for $\crrjc$ and $\crdjc$ will be similar).
We show that $\deg H \le 1$. 

For that purpose, notice that above the anti-diagonal,
$\jac H$ is anti-symmetric with respect to the diagonal. Now the 
assumption that $x_1$ appears above the anti-diagonal in $\jac H$
implies that $H_j$ has a term divisible by $x_1 x_k$ for some $j, k$ with 
$j + k \le n + 1$, and leads to a contradiction in a similar manner as in the 
proof of theorem \ref{asymjc} (with $i = 1$), because the argument in this 
proof does not get below the anti-diagonal of $\jac H$, and stays in the part 
where $\jac H$ is anti-symmetric.

Thus there is no $x_1$ above or on the anti-diagonal of $\jac H$. By the 
\antisym\ of the anti-diagonal, there is no $x_1$ in $\jac H$. 
Consequently, the first column of $\jac H$ is constant. 
By the \antisym\ conditions, all border entries of $\jac H$
are constant. Hence the entries of $\jac H$ that are not
on the border do not contain $x_1$ and neither $x_n$, and form 
a matrix with the same \antisym\ as $\jac H$ itself. So by 
induction on $n$, it follows that $\deg H \le 1$. Hence $F = x + H$ satisfies
the Jacobian conjecture, as desired.
\end{proof}

\noindent
In theorem \ref{sdjck} in the next section, we show that $\sjc(K,n)$, $\crjc(K,2n-1)$ and
$\crjc(K,2n)$ are equivalent when $K$ is a field of characteristic zero.

\section{Centrally symmetric variants of the Jacobian conjecture} \label{sec2}

\begin{definition}
$\hvjc(K,n)$, $\havjc(K,n)$, $\ahvjc(K,n)$, $\ahavjc(K,n)$ have 
horizontal and vertical \antisyms\ in their definitions.

$\hvsjc(K,n)$, $\havsjc(K,n)$, $\ahvsjc(K,n)$, $\ahavsjc(K,n)$ have 
horizontal, vertical, and diagonal \antisyms\ in their 
definitions.

$\djc(K,n)$ means that the Jacobian conjecture is satisfied for $n$-dimensional
maps $F = x + H$ as above that have Jacobians that are symmetric with respect to the center, i.e.,
entries $(i,j)$ and $(n+1-i,n+1-i)$ of $\jac H$ are equal for all $i,j$.

In the definition of $\adjc(K,n)$, $\cjc(K,n)$ and $\acjc(K,n)$, the central point 
`symmetry' is at least partially an antisymmetry.
\end{definition}

\noindent
Notice that for any matrix $A \in \Mat_n(K)$, the matrix
$$
M := \left( \begin{array}{cc} +A & \pm A \\ \mp A & -A \end{array} \right)
$$
of size $2n$ has some sort of tiling \antisym. By conjugating the map 
$(x,y) \mapsto M(x,y)$ with $(x,y\rv)$, we get a map with horizontal and vertical
\antisyms, since
$$
\left( \begin{array}{cc} I_n & {\mathbf 0} \\ {\mathbf 0} & I_n\rv \end{array} \right) M
\left( \begin{array}{cc} I_n & {\mathbf 0} \\ {\mathbf 0} & I_n\rv \end{array} \right) =
\left( \begin{array}{cc} +A & \pm((A\tp)\rv)\tp \\ \mp A\rv & -(((A\tp)\rv)\tp)\rv \end{array} \right)
= \left( \ifUseEepic \begin{array}{cc} 
    +A & \pm\flipv{A} \\ \mp\fliph{A} & -\rotu{A} 
\end{array} \else \begin{array}{cc} 
    \begin{tikzpicture}[x=1pt,y=1pt] \useasboundingbox (-7.5,-3) rectangle (7.5,5); 
    \path (0,0) node {$+A$}; \end{tikzpicture} &
    \begin{tikzpicture}[x=1pt,y=1pt] \useasboundingbox (-7.5,-3) rectangle (7.5,5); 
    \path (0,0) node[xscale=-1] {$A \pm$}; \end{tikzpicture} \\
    \begin{tikzpicture}[x=1pt,y=1pt] \useasboundingbox (-7.5,-3) rectangle (7.5,5); 
    \path (0,0) node[yscale=-1] {$\pm A$}; \end{tikzpicture} &
    \begin{tikzpicture}[x=1pt,y=1pt] \useasboundingbox (-7.5,-3) rectangle (7.5,5); 
    \path (0,0) node[xscale=-1,yscale=-1] {$A-$}; \end{tikzpicture}
\end{array} \fi \right)
$$

Any square matrix $M$ can be written as
$$
M = \frac12 (M + M\tp) + \frac12 (M - M\tp)
$$ 
which is the sum of a matrix with regular symmetry $\sjc$ 
and a matrix with regular antisymmetry $\asjc$. In a similar manner, a
square matrix $M$ of size $N$ with symmetry $\djc$ can be written as
$$
M = \frac12 (M + M\rv) + \frac12 (M - M\rv)
$$
which is the sum of a matrix with symmetry $\hvjc$ and a matrix with
\antisym\ $\ahavjc$, because the symmetry $\djc$ means exactly that
$M\rv = M I_N\rv$, and the left and right hand side of $M \pm M\rv = M \pm M I_N\rv$ 
have a horizontal and a vertical \antisym\ axis respectively.

If $F = X + H \in K[X]^N$ such that $\jac H$ has symmetry $\djc$,
then we can write
\begin{equation} \label{dotdecom}
F = \frac12 (F + F\rv) + \frac12 (F - F\rv)
\end{equation}
and $\frac12 (F + F\rv)$ and $\frac12 (F - F\rv)$ have Jacobian \antisyms\
$\hvjc$ and $\ahavjc$ respectively, because $\jac_X (F\rv) = (\jac_X F)\rv$.

\begin{theorem} \label{djck}
Assume $K$ is a field of characteristic zero. 
Then $\jc(K,n)$, $\djc(K,2n-1)$, $\djc(K,2n)$, $\hvjc(K,2n-1)$, 
$\hvjc(K,2n)$, $\ahavjc(K,2n)$ and $\ahavjc(K,\allowbreak 2n+1)$ 
are all equivalent. 
\end{theorem}

\begin{proof}
Since $\jc(K,n+1)$ implies $\jc(K,n)$, it suffices to prove the following.
\begin{equation} \label{eqv1} \begin{split}
\hvjc(K,2n) &\Longleftrightarrow \jc(K,n) \\
\ahavjc(K,2n) &\Longleftrightarrow \jc(K,n) \qquad\qquad \mbox{and} \\
\djc(K,2n) &\Longleftrightarrow \jc(K,n) \wedge \jc(K,n)
\end{split} \end{equation}
and
\begin{equation} \label{eqv2} \begin{split}
\hvjc(K,2n+1) &\Longleftrightarrow \jc(K,n+1) \\
\ahavjc(K,2n+1) &\Longleftrightarrow \jc(K,n) \qquad\qquad \mbox{and}  \\
\djc(K,2n+1) &\Longleftrightarrow \jc(K,n+1) \wedge \jc(K,n)
\end{split} \end{equation}
We only prove \eqref{eqv2}, since \eqref{eqv1} can be proved in a similar 
manner: you just ignore the $(n+1)$-th row and column.
\begin{description}

\item[(\RightArroW):] 
Let $(F,f) = (x,x_{n+1}) + (H,h)$ be a polynomial map in dimension $n+1$, where $f = x_{n+1} + h$ 
is a single polynomial and $F = x + H$ is an $n$-tuple of polynomials. Let
$\tilde{F} = x + \tilde{H}$ be a polynomial map in dimension $n$.
Then $(F, f, \tilde{F}(y))$ is invertible or of Keller type, if and only if both $(F,f)$ and $\tilde{F}$ 
are invertible or of Keller type respectively.

By conjugating
$(F, f, \tilde{F}(y))$ with the linear map $(x+y,x_{n+1},x-y)$, we obtain
\begin{align}
\frac12 \left( \begin{array}{ccc} 
I_n & {\mathbf 0} & I_n \\
{\mathbf 0} & 2 & {\mathbf 0} \\
I_n & {\mathbf 0} & -I_n \end{array} \right) &
\left( \begin{array}{c} F(x+y,x_{n+1}) \\ f(x+y,x_{n+1}) \\
\tilde{F}(x-y) \end{array} \right) \nonumber \\
&= \frac12 \left( \begin{array}{c} F(x+y,x_{n+1}) + \tilde{F}(x-y) \\ 
2f(x+y,x_{n+1}) \\ F(x+y,x_{n+1}) - \tilde{F}(x-y) \end{array} \right) \label{FftF}
\end{align}
Now the Jacobian of the $F$-part and the $\tilde{F}$-part of the right hand side
of \eqref{FftF}, without the row and column in the middle, 
have tiling \antisyms
$$
\left( \begin{array}{cc} +A & +A \\ +A & +A \end{array} \right)
\qquad \mbox{and} \qquad
\left( \begin{array}{cc} +A & -A \\ -A & +A \end{array} \right)
$$
respectively. A subsequent conjugation with $(x, x_{n+1}, y\rv)$ 
of the right hand side of \eqref{FftF} gives Jacobian symmetry $\hvjc$ 
for the $(F,f)$-part, and Jacobian \antisym\ $\ahavjc$ 
for the $\tilde{F}$-part. 

Since $\djc$ is a subsymmetry of both $\hvjc$ and $\ahavjc$, we have a
map with symmetry $\djc$ in general, a map with symmetry $\hvjc$ when $\tilde{F}(y) = y$, 
and a map with \antisym\ $\ahavjc$ when $(F,f) = (x,x_{n+1})$. This map is a conjugation
of $(F, f, \tilde{F}(y))$, and the forward implications in \eqref{eqv2} follow
because conjugations preserve the linear part $(x,x_{n+1},y)$ of
$(F, f, \tilde{F}(y))$.

\item[(\LeftArroW):]
Take $G \in K[x,x_{n+1},y]^{2n+1}$ such that $G - (x,x_{n+1},y)$ has Jacobian symmetry $\djc$.
By \eqref{dotdecom}, we can write $G = \frac12(G + G\rv) + \frac12(G - G\rv)$, where

\begin{align*}
G + G\rv &= 
\left( \begin{array}{c} F(x + y\rv, x_{n+1}, x - y\rv) \\ 
       2f(x + y\rv, x_{n+1}, x - y\rv) - x_{n+1} \\
       F\rv(x + y\rv, x_{n+1}, x - y\rv) \end{array} \right)
\intertext{has Jacobian symmetry $\hvjc$, and }
G - G\rv &= 
\left( \begin{array}{c} 
       \tilde{F}(x + y\rv, x_{n+1}, x - y\rv) \\ 
       x_{n+1} \\
       - \tilde{F}\rv(x + y\rv, x_{n+1}, x - y\rv) \end{array} \right)
\end{align*}
has Jacobian \antisym\ $\ahavjc$. From the \antisyms\ $\hvjc$ and $\ahavjc$, we can 
derive that $F_j,f \in K[x,x_{n+1}]$ and $\tilde{F}_j \in K[y]$ for all $j \le n$.
By replacing $\tilde{F}(x,x_{n+1},y)$ by $\tilde{F}(y)$ (and $F(x,x_{n+1},y)$ by $F(x,x_{n+1})$), 
we obtain that $G$ is of the form
$$
G = \frac12(G + G\rv) + \frac12(G - G\rv) = \frac12\left( \begin{array}{c} 
       F(x + y\rv, x_{n+1}) + \tilde{F}(x - y\rv) \\ 
       2f(x + y\rv, x_{n+1}) \\
       F\rv(x + y\rv, x_{n+1}) - \tilde{F}\rv(x - y\rv) 
\end{array} \right)
$$
Hence the conjugation of $G$ with the linear map $(x,x_{n+1},y\rv)$ is equal to the 
right hand side of \eqref{FftF}, which is the conjugation of $(F, f, \tilde{F}(y))$
with the linear map $(x+y,x_{n+1},x-y)$.  This gives the last backward implication in 
\eqref{eqv2}. The other backward implications in \eqref{eqv2} follow by taking 
$\tilde{F}(y) = y$ and $(F,f) = (x,x_{n+1})$ respectively.
\qedhere

\end{description}
\end{proof}

\begin{definition}
Define an `instance of $\jcs(K,n)$' as a Keller map $F \in K[x]^n$ such that $F - x$
has Jacobian \antisym\ $\jcs$, where $\jcs$ is any \antisym\ of the square.
\end{definition}

\noindent
Notice that in both \eqref{eqv1} and \eqref{eqv2}, the first and the second right hand side are satisfied,
if and only if the last right hand side is satisfied. Hence we have proved the following as well.

\begin{theorem}
Assume $K$ is a field of characteristic zero. 
A map $G = X + H$ is an (invertible) instance of $\djc(K,N)$, 
if and only if $X + \frac12(H + H\rv)$ is an (invertible) instance of 
$\hvjc(K,N)$ and $X + \frac12(H - H\rv)$ is an (invertible) 
instance of $\ahavjc(K,N)$.
\end{theorem}

\noindent
When we combine horizontal and vertical \antisyms\ with diagonal
ones, we get the following. 

\begin{theorem} \label{sdjck}
Assume $K$ is a field of characteristic zero. Then
$\sjc(K,n)$, $\crjc(K,2n-1)$, $\crjc(K,2n)$, $\hvsjc(K,2n-1)$, $\hvsjc(K,2n)$,  
$\ahavsjc(K,2n)$ and $\ahavsjc(K,\allowbreak 2n+1)$ are all equivalent.
\end{theorem}

\begin{proof}
The proof is similar to that of the equivalence of $\jc(K,n)$, $\djc(K,2n-1)$, 
$\djc(K,2n)$, $\hvjc(K,2n-1)$, $\hvjc(K,2n)$, $\ahavjc(K,2n)$ and $\ahavjc(K,2n+1)$
in the proof of theorem \ref{djck}.
\end{proof}

\noindent
Notice that we have proved the following as well.

\begin{theorem}
Assume $K$ is a field of characteristic zero. 
Then a map $G = X + H$ is an (invertible) instance of $\crjc(K,N)$, 
if and only if $X + \frac12(H + H\rv)$ is an (invertible) instance of 
$\hvsjc(K,N)$ and $X + \frac12(H - H\rv)$ is an (invertible) 
instance of $\ahavsjc(K,N)$.
\end{theorem}

\noindent
Now assume that $F = x + H$ is power linear of {\em even} degree. Then the 
construction of an instance of $\hvjc(K,2n)$ out of the instance $F$ of
$\jc(K,n)$ gives a map that is power linear of even degree again, say
$(x,y) + (B(x,y))^{*d}$. Since $d$ is even, we can assume that $B$ has \antisym\
$\ahvjc$ instead of $\hvjc$. But that means that $B^2 = 0$. 

In the general case, we can make $(x,y) + (B(x,y))^{*d}$ out of 
$F = x + (Ax)^{*d}$, where
$$
B := \left( \begin{array}{cc} ab & -b^2 \\ a^2 & -ab \end{array} \right) \otimes A
= \left( \begin{array}{cc} ab A & -b^2 A \\ a^2 A & -ab A \end{array} \right)
$$
and again we have $B^2 = 0$ because the left factor of the Kronecker tensor product 
squares to zero as well. More precisely, if
$$
T := \left( \begin{array}{cc}
a\!\sqrt[d-1]{a b^d - a^d b} & -b\!\sqrt[d-1]{a b^d - a^d b} \\
a^d & -b^d \end{array} \right) \otimes I_n
$$
then the determinant of the left factor of the Kronecker product
just above is $-(\!\sqrt[d-1]{a b^d - a^d b})^d$. By Cramer's rule,
$$
T^{-1} = \bigg(\left(\frac1{\!\!\sqrt[d-1]{a b^d - a^d b}}\right)^{\!d} \cdot 
\left( \begin{array}{cc} b^d & -b\!\sqrt[d-1]{a b^d - a^d b} \\
a^d & -a\!\sqrt[d-1]{a b^d - a^d b} \end{array} \right)\bigg) \otimes I_n
$$
Since $\big((Ax)^{*d},0\big) = (1,0) \otimes (Ax)^{*d}$, it follows from the mixed product property
of the Kronecker product and the regular matrix product that
\begin{alignat*}{2}
T^{-1} (F,y)|_{(x,y)=T(x,y)}
&= (x,y) + {}&& \left(T^{-1} \cdot \left( \begin{array}{c} 1 \\ 0 \end{array} \right)\right)
   \otimes \big(Ax\big)^{*d}\big|_{(x,y) = T(x,y)} \\
&= (x,y) + {}&& \bigg(\left(\frac1{\!\!\sqrt[d-1]{a b^d - a^d b}}\right)^{\!d} \cdot 
\left( \begin{array}{cc} b^d \\ a^d \end{array} \right)\bigg) \otimes \\ & {} &&
\Big( A \big( a\!\sqrt[d-1]{a b^d - a^d b} x - b\!\sqrt[d-1]{a b^d - a^d b} y \big) \Big)^{*d} \\
&= (x,y) + {}&& \left( \begin{array}{c} b^d \\ a^d \end{array} \right) \otimes 
\Big( A \big( ax - by \big) \Big)^d \\
&= (x,y) + {}&& \Big(B(x,y)\Big)^{*d}
\end{alignat*}
See also \cite{druzA2}.

\begin{theorem} \label{nplusone}
Assume $K$ is a field of characteristic zero. 
Then $\cjc(K,2n) \Leftrightarrow \cjc(K,2n+1)$ and
$\acjc(K,2n) \Leftrightarrow \acjc(K,2n+1)$ hold.
\end{theorem}

\begin{proof}
Both equivalences follow by stabilization: the row and column in the middle of a 
matrix with \antisym\ $\cjc$ or $\acjc$ are both zero.
\end{proof}

\noindent
The following theorem shows that complex polynomial maps can be seen as real 
polynomial maps with a certain Jacobian \antisym.

\begin{theorem} \label{cjcr}
$\jc(\C,n)$, $\djc(\C,2n-1)$, $\djc(\C,2n)$, $\cjc(\C,2n)$, $\cjc(\C,2n+1)$, 
$\cjc(\R,2n)$ and $\cjc(\R,2n+1)$ are equivalent.
\end{theorem}

\begin{proof}
The equivalence of $\jc(\C,n)$, $\djc(\C,2n-1)$ and $\djc(\C,2n)$ 
follows from theorem \ref{djck}. The equivalence of $\cjc(K,2n)$ and $\cjc(K,2n+1)$
for $K \in \{\C,\R\}$ follows from theorem \ref{nplusone}. 
Since the implication $\djc(\C,2n) \Rightarrow 
\cjc(\C,2n)$ follows by conjugation with $(x, \I y)$ and the implication 
$\cjc(\C,2n) \Rightarrow \cjc(\R,2n)$ is direct, the implication
$\cjc(\R,2n) \Rightarrow \jc(\C,n)$ remains to be proved.

Assume that $\cjc(\R,2n)$ holds. Let $F$ be an instance of Keller type of $\jc(\C,n)$. 
Then $(F(x), \tilde{F}(y))$ is of Keller type as well, where the coefficients
of the polynomial map $\tilde{F}$ are the complex conjugates of those of $F$.
Furthermore, $F$ is invertible in case $(F(x), \tilde{F}(y))$ is.

We prove that $\jc(\C,n)$ holds by showing that $(F(x), \tilde{F}(y))$ can be 
transformed to an instance of $\cjc(\R,2n)$ by compositions with invertible 
linear maps. By \cite[Prop.\@ 1.1.7]{arnoboek}, the instance of $\cjc(\R,2n)$
satisfies the Jacobian conjecture over $\C$ as well as over $\R$, which gives
$\jc(\C,n)$.

Notice that
\begin{align}
\frac12 \left( \begin{array}{cc} I_n & I_n \\
 - \I I_n & \I I_n \end{array} \right) 
\left( \begin{array}{c} F(x + \I y) \\ \tilde{F}(x - \I y) \end{array} \right)
&= \frac12 \left( \begin{array}{c} F(x + \I y) + \tilde{F}(x - \I y) \\ 
- \I F(x + \I y) + \I \tilde{F}(x - \I y) \end{array} \right) 
\nonumber \\
&= \left( \begin{array}{c} \Rea\,F(x + \I y) \\ \Ima\,F(x + \I y) \end{array} \right) 
\label{complexharm}
\end{align}
if $x$ and $y$ are considered as real variables, and that \eqref{complexharm} is a polynomial
map with real coefficients by definition of $\tilde{F}$. Since the Jacobian of \eqref{complexharm} is
$$
\frac12 \left( \begin{array}{cc}
(\jac F)|_{x = x + \I y} + (\jac \tilde{F})|_{x = x - \I y} &
\I (\jac F)|_{x = x + \I y} - \I (\jac \tilde{F})|_{x = x - \I y} \\
-\I (\jac F)|_{x = x + \I y} + \I (\jac \tilde{F})|_{x = x - \I y} &
(\jac F)|_{x = x + \I y} + (\jac \tilde{F})|_{x = x - \I y}
\end{array} \right)
$$
which has tiling \antisym
$$
\left( \begin{array}{cc} A & B \\ -B & A \end{array} \right)
$$ 
a conjugation of \eqref{complexharm} with $(x,y\rv)$ gives a map with Jacobian
\antisym\ $\cjc$. If we start with $(F(x),\tilde{F}(y)) = (x,y)$, then this map is 
equal to $(x,y)$ as well, thus the compositions with linear maps add up to a linear 
conjugation. Hence the conjugation of \eqref{complexharm} with $(x,y\rv)$ gives 
an instance of $\cjc(\R,2n)$. 
\end{proof}

\subsection*{Symmetry patterns that satisfy the Jacobian Conjecture}

\begin{theorem}
Assume $K$ is a field of characteristic zero. Then
$\havjc(K,N)$, $\ahvjc(K,N)$, $\havsjc(K,N)$, 
and $\ahvsjc(K,N)$ have affirmative answers.
\end{theorem}

\begin{proof}
Let $F = X + H$ be an instance of any of them. One can easily see that
$\jac H \cdot H$ is a vector for which all of its $N$ coordinates
are sums of $N$ terms that cancel out pairwise except maybe the term in the middle if $N$ is odd.
But that term can only be zero, since either the column in the middle of $\jac H$ is zero or
the component in the middle of $H$ is zero, depending on the actual \antisym.
Therefore,
\begin{equation} \label{quasisym}
\jac H \cdot H = 0
\end{equation}
Now it follows from ii) $\Rightarrow$ i) in \cite[Prop.\@ 1.1]{art10} that
$G = x - H$ is the inverse of $F$. 
\end{proof}

\noindent
Maps $x + H$ with inverse $x - H$ are called quasi-translations.
Over a field of charateristic zero, these are exactly the maps that
satisfy \eqref{quasisym}. See also \cite[Prop.\@ 1.1]{art10}. 
Quasi-translations arise naturally with singular Hessians:
if $\det \hess h = 0$, then there exists a nonzero polynomial $R$
such that $R(\grad h) = 0$, and $x + (\grad R)(\grad h)$ happens to
be a quasi-translation on account of \cite[Eq.\@ (3)]{art3} and
\cite[Prop.\@ 1.1]{art10}.

\section{Symmetric variants of the dependence problem} \label{sec3}

This section is about polynomial maps $H$ instead of $F = x + H$,
again with a certain \antisym\ in the Jacobian of the map $H$, but now the question is whether
the (linear) dependence problem (for Jacobians) is satisfied
for such maps. We say that $H \in K[x]^n$ satisfies the dependence problem if
$$
\lambda\tp \jac H = 0
$$
for some nonzero $\lambda \in K^n$, which is equivalent to
$$
\lambda_1 H_1 + \lambda_2 H_2 + \cdots + \lambda_n H_n \in K
$$
when $K$ is a field of characteristic zero. 
Notice that replacing $H$ by any composition of $H$ with invertible
linear maps does not change whether $H$ satisfies the dependence problem.

One may think that $\det \jac H = 0$ is a somewhat weak condition for getting
linear dependence. Indeed, the dependence problem with $\det \jac H = 0$
is satisfied for arbitrary Jacobians only in dimension $1$
and for homogeneous Jacobians only in dimensions $1$ and $2$.
But the dependence problem with $\det \jac H = 0$ and additionally $\jac H$ 
symmetric with regular symmetry $\sjc$ reaches twice as far: it is satisfied 
for arbitrary Jacobians in dimensions $1$ and $2$, and for homogeneous Jacobians 
in all dimensions $n \le 4$, see \cite{dillen} and \cite{gornoet} respectively. 

Corollary \ref{hessreddp} below shows that the above results about symmetric Jacobians
imply those about non-symmetric Jacobians. By theorem \ref{crdp} below, we see that 
the dependence problem with $\det \jac H = 0$ and additionally $\jac H$ 
symmetric with symmetry $\crjc$ is even satisfied for arbitrary Jacobians in 
dimensions $n \le 5$, and for homogeneous Jacobians in all dimensions $n \le 9$. 

\begin{definition}
$\jc[K,n]$ means that the dependence problem is satisfied for 
$n$-dimensional maps $H$ over the field $K$, such that the degree of each term of $H$
is contained in a fixed set $S \subseteq \N$, and optionally $\det \jac H = 0$ 
or even $(\jac H)^n = 0$.

$\sjc[K,n]$ and $\rsjc[K,n]$ mean that the dependence problem is
satisfied for $n$-dimensional maps $H$ as above that have a symmetric Jacobian 
with respect to the diagonal and anti-diagonal respectively.

Et cetera. We replace the parenthesis of the symmetric variants of
the Jacobian conjecture by square brackets all the time.
\end{definition}

\noindent
Notice that a horizontal symmetry axis in the Jacobian of $H$ implies 
linear dependence over the base field $K$ between the rows of $\jac H$ (in case
the dimension is larger than one).
Therefore, such \antisyms\ will not be considered any further in this 
section. The same holds for (partial) antisymmetries that imply that the
row in the middle of $\jac H$ is zero.

\subsection*{Diagonally symmetric variants}

\begin{definition}
Define an `instance of $\jcs[K,n]$' as a map $H \in K[x]^n$ such that $\jac H$
has \antisym\ $\jcs$, where $\jcs$ is any \antisym\ of the square.
\end{definition}

\begin{theorem} \label{menghessdp}
Assume $K$ is a field of characteristic zero. 
Then $\rsjc[K,2n]$ implies $\jc[K,n]$.
\end{theorem}

\begin{proof}
If the set $S$ of term degrees
in the definition of $\jc[K,n]$ satisfies $S \subseteq \{0,1\}$, then 
all that matters for $\jc[K,n]$ is whether $\det \jac H$ vanishes or not in the definition of 
$\jc[K,n]$, and theorem \ref{menghessdp} is trivially satisfied by definition. 
Thus assume that $d \in S$ for some $d \ge 2$.

Suppose that $H$ is an instance of $\jc[K,n]$. With the proof of theorem \ref{menghess}, we
obtain a map $(H(x),G(x,y))$ which is an instance of $\rsjc[K,2n]$, and 
$$
\jac \left( \begin{array}{c} H(x) \\ G(x,y) \end{array} \right) =
\left( \begin{array}{cc} 
    \jac H & {\mathbf 0} \\
    * & (((\jac H)\rv)\tp)\rv
\end{array} \right) = \left( \ifUseEepic \begin{array}{cc} 
    \rotz{\jac H} & {\mathbf 0} \\
    ^{\rotz{*}} & \flipr{\!\!\!\jac H~} 
\end{array} \else \begin{array}{cc}
    \jac H & {\mathbf 0} \\ 
    \begin{tikzpicture}[x=1pt,y=1pt] \useasboundingbox (-5,-5) rectangle (5,10);           
    \path (0,0) node{$*$}; \end{tikzpicture} &
    \begin{tikzpicture}[x=1pt,y=1pt] \useasboundingbox (-5,-5) rectangle (5,10);           
    \path (0,0) node[xscale=-1,rotate=90]{$\jac H$}; \end{tikzpicture}
\end{array} \fi \right)
$$
Since the characteristic polynomial of the above matrix is the square of that
of $\jac H$, the nilpotency (or the vanishing of the determinant) of the 
Jacobian of $(H(x),G(x,y))$ is completely determined by the nilpotency 
(or the vanishing of the determinant) of $\jac H$.
So we get another instance of $\rsjc[K,2n]$ if we change $G$ such that only the $*$-part
of the above Jacobian changes. Hence we may change the terms without $y$ in $G$
by other terms without $y$, but we must not forget to preserve the symmetry $\rsjc$
of the $*$-part.
We do this by replacing the part of $G$ that
has terms without $y$ only by $(x\rv)^{*d} =
(x_n^d, x_{n-1}^d, \ldots, x_2^d, x_1^d)$.

Now assume that $\rsjc[K,2n]$ is satisfied. Then the components of $(H,G)$ 
are linearly dependent over $K$, say that
$$
\lambda\tp H + \mu\tp G \in K
$$
where $\lambda, \mu \in K^n$ are not both zero. If $\mu = 0$, then the components
of $H$ are linearly dependent over $K$, as desired, so it suffices to show that
$\mu = 0$.  Since $H$ has no terms with $y$, $\lambda\tp \jac_y H = 0$
and we obtain that $\mu\tp \jac_y G = 0$, too.
Thus $\lambda\tp \jac_y H \lambda\rv = \mu\tp \jac_y G \lambda\rv = 0$, 
and by the symmetry $\rsjc$ of $\jac_{x,y} (H,G)$ and by \eqref{mrev}, we see that
$$
\lambda\tp \jac_x H \mu\rv 
= \lambda\tp (I_{2n}\rv \jac_x\rv H) I_{2n}\rv \mu  
= \big(\mu\tp I_{2n}\rv (\jac_x\rv H)\tp \lambda\rv) \big) \tp 
= \big(\mu\tp \jac_y G \lambda\rv\big) \tp = 0
$$
Hence by $\lambda\tp H + \mu\tp G \in K$ and \eqref{mrev},
$$
0 = \big(\lambda\tp \jac_x H \mu\rv + \mu\tp \jac_x G \mu\rv\big)\big|_{y=0}
= d \mu\tp I_n\rv\diag\big(x^{*(d-1)}\big) \mu\rv
= d \sum_{i=1}^n \mu_{n+1-i}^2 x_i^{d-1}
$$
which is a contradiction to $\mu \ne 0$ because $d \ge 2$.
\end{proof}

\begin{theorem} \label{hessequivdp}
$\sjc[\C,N]$ and $\rsjc[\C,N]$ are equivalent.
\end{theorem}

\begin{proof}
The proof is similar to that of theorem \ref{hessequiv}.
\end{proof}

\begin{corollary} \label{hessreddp}
$\sjc[\C,2n]$ implies $\jc[\C,n]$.
\end{corollary}

\begin{proof}
This follows from theorem \ref{hessequivdp} and theorem \ref{menghessdp}.
\end{proof}

\noindent
Notice that there is no converse of theorems \ref{menghess} and \ref{menghessdp}.
But if we define $\rsnjc(K,n)$ and $\rsnjc[K,n]$ as $\rsjc(K,n)$ and $\rsjc[K,n]$
respectively with the extra condition that the upper right quadrant of
the Jacobian is zero, then we do have a converse. 

\subsection*{Centrally symmetric variants}

In a similar way as above,
$\adjc(K,N)$ and $\adjc[K,N]$ are equivalent to 
$\axjc(K,N)$ and $\axjc[K,N]$ respectively. The proof is left as an exercise to the reader.

The following theorem is an analog of theorem \ref{djck}, but
one of the indexes of $\djc$ is different: $2n+1$ instead of $2n-1$.
This is because the $\wedge$'s become $\vee$'s. But first, we formulate a lemma.

\begin{lemma} \label{nred}
Assume $K$ is a field of characteristic zero. Then $\jc[K,n+1]$ implies $\jc[K,n]$.
\end{lemma}

\begin{proof} 
Assume $H$ is an instance of $\jc[K,n]$. If the set $S$ in the definition of $\jc[K,n]$ 
is a subset of $\{0\}$, then $H$ satisfies the dependence problem. Thus assume 
$d \in S$ for some $d \ge 1$. If $\det \jac H = 0$ by definition of $\jc[K,n]$, then
there exists an $i \le n$ such that $e_i\tp$ is not contained in the row space (over $K(x)$)
of $\jac H$. We take $h = x_i^d$ in case 
$\det \jac H = 0$ by definition of $\jc[K,n]$ and $h = x_{n+1}^d$ otherwise (because 
we need $\parder{}{x_{n+1}} h = 0$ when $\jac H$ is nilpotent by definition of $\jc[K,n]$).
Then $(H,h)$ is an instance of $\jc[K,n+1]$ and 
$H$ satisfies the dependende problem in case $(H,h)$ does, as desired.
\end{proof}

\begin{theorem}
Assume $K$ is a field of characteristic zero. 
Then $\jc[K,n]$, $\djc[K,2n]$, and $\djc[K,2n+1]$ are equivalent.
\end{theorem}

\begin{proof}
By lemma \ref{nred}, it suffices to prove that
\begin{align*}
\djc[K,2n] &\Longleftrightarrow \jc[K,n] \vee \jc[K,n]
\intertext{and}
\djc[K,2n+1] &\Longleftrightarrow \jc[K,n+1] \vee \jc[K,n]
\end{align*}
This can be done in the same way as
\begin{align*}
\djc(K,2n) &\Longleftrightarrow \jc(K,n) \wedge \jc(K,n)
\intertext{and}
\djc(K,2n+1) &\Longleftrightarrow \jc(K,n+1) \wedge \jc(K,n)
\end{align*}
in the proof of theorem \ref{djck}.
\end{proof}

\begin{theorem} \label{crdp}
Assume $K$ is a field of characteristic zero. Then
$\sjc[K,n]$, $\crjc[K,2n]$ and $\crjc[K,2n+1]$ are equivalent.
\end{theorem}

\begin{proof}
The proof is similar to that of the previous theorem.
\end{proof}

\begin{theorem} \label{cjcrdp}
$\jc[\C,n]$, $\djc[\C,2n]$, $\djc[\C,2n+1]$, $\cjc[\C,2n]$ and $\cjc[\R,2n]$ 
are equivalent.
\end{theorem}

\begin{proof}
The proof is similar to that of theorem \ref{cjcr}.
\end{proof}

\begin{corollary}
$\djc[\R,4n]$ implies $\jc[\C,n]$.
\end{corollary}

\begin{proof}
$\djc[\R,4n] \Leftrightarrow \jc[\R,2n] \Rightarrow \cjc[\R,2n] 
\Leftrightarrow \jc[\C,n]$.
\end{proof}

\subsection*{Planar singular Hessians and planar nilpotent Jacobians}

If we assume that $\det \jac H = 0$ instead of that $\jac H$ is nilpotent, then we can
transform \antisyms\ more freely, because we do not need to conjugate. We can just compose
with maps in $\GL_n(K)$, where $K$ is a field of characteristic zero. 
Or with maps in $\GL_n(A)$, if we replace $K$ by an integral domain $A$. 
Now if $\jac H$ has \antisym\ $\trasjc$, then
its trace is zero. So if $\jac H$ has dimension $2$, then $\jac H$ is nilpotent, if
and only if $\det \jac H = 0$ and $\jac H$ has \antisym\ $\trasjc$. We can use this
observation to prove the following (which can also be derived directly from 
\cite[Th.\@ 3.1]{art3}, see \cite[Cor.\@ 5.1.2]{homokema}).

\begin{theorem} \label{hessdim2ufd}
Assume $A$ is a unique factorization domain of characteristic zero and 
$h \in A[x_1,x_2]$ such that $\det \hess h = 0$. Then $h$ is of the form
$$
g(a x_1 - b x_2) + (c x_1 - d x_2)
$$
where $g $ is an univariate polynomial over $A$ and $a,b,c,d \in A$. 
Furthermore, $g$ can be taken constant in case $\rk \hess h = 0$.
\end{theorem}

\begin{proof}
Let $H := \grad h$. Since $\jac H$ has symmetry $\sjc$, we see that
$\jac (-H_2,H_1)$ has \antisym\ $\trasjc$, thus $\tr \jac (-H_2,H_1) = 0$.
Since $\det \jac (-H_2,H_1) = 0$ as well, we get that $\jac (-H_2,H_1)$ is
nilpotent. By \cite[Th.\@ 7.2.25]{arnoboek} (see also \cite{esshubnc}), we obtain
that $(-H_2,H_1)$ is of the form
$$
\left( \begin{array}{c} -H_2 \\ H_1 \end{array} \right) =
\left( \begin{array}{c} b g'(a x_1 - b x_2) + d \\ 
       a g'(a x_1 - b x_2) + c \end{array} \right)
$$
where $g'$ is the derivative of an univariate polynomial $g$ over $A$.
Now one can easily see that $H$ is the gradient map of 
$g(a x_1 - b x_2) + (c x_1 - d x_2)$. Hence $h$ is of the desired form.
\end{proof}

\section*{Conclusion}

We have seen that besides the regular diagonal symmetry, there are many other
interesting \antisym\ properties, over which a lot can be said in connection with the 
Jacobian conjecture.


\begin{thebibliography}{99}

\bibitem{art3}
Bondt M.C. de, Essen A.R.P. van den,
Singular Hessians,
J. Algebra 282 (2004), no.\@ 1, 195--204

\bibitem{art6}
Bondt M.C. de, Essen A.R.P. van den,
A reduction of the Jacobian conjecture to the symmetric case,
Proc.\@ Amer.\@ Math.\@ Soc.\@ 133 (2005), no.\@ 8, 2201--2205 (electronic)

\bibitem{art10}
Bondt M.C. de,
Quasi-translations and counterexamples to the homogeneous dependence problem,
Proc.\@ Amer.\@ Math.\@ Soc.\@ 134 (2006), no.\@ 10, 2849--2856 (electronic)

\bibitem{dillenext}
Bondt M.C. de,
Constant polynomial Hessian determinants in dimension three,
arXiv:1203.6605, 2012

\bibitem{dillen}
Dillen F.,
Polynomials with constant Hessian determinant.
J. Pure Appl.\@ Algebra 71 (1991), no.\@ 1, 13--18

\bibitem{gornoet}
Gordan P., N{\"o}ther M.,
Ueber die algebraischen Formen, deren Hesse'sche Determinante
identisch verschwindet,
Math.\@ Ann.\@ 10 (1876), no.\@ 4, 547--568

\bibitem{homokema}
Bondt M.C. de,
Homogeneous Keller maps, Ph.D. thesis, Radboud University Nijmegen, July 2009, \\
\verb+http://webdoc.ubn.ru.nl/mono/b/bondt_m_de/homokema.pdf+

\bibitem{druzA2} 
Dru{\.z}kowski L.M., New reduction in the Jacobian conjecture,
Effective methods in algebraic and analytic geometry, 2000 (Krak{\'o}w),
Univ.\@ Iagel.\@ Acta Math.\@ No.\@ 39 (2001), 203--206 

\bibitem{druzsym}
Dru{\.z}kowski L.M., The Jacobian conjecture: symmetric reduction and 
solution in the symmetric cubic linear case,
Ann.\@ Polon.\@ Math. 87 (2005), 83--92

\bibitem{arnoboek}
Essen A.R.P. van den, Polynomial Automorphisms and the Jacobian Conjecture,
Vol.\@ 190 in Progress in Mathematics, Birkh{\"a}user, 2000

\bibitem{esshubnc}
Essen A.R.P. van den, Hubbers E.-M.G.M., 
A new class of invertible polynomial maps,
J. Algebra 187 (1997), no.\@ 1, 214-226

\bibitem{meng} 
Meng G., Legendre transform, Hessian conjecture and tree formula.
Appl.\@ Math.\@ Lett., 19 (2006), no.\@ 6, 503-510.

\end{thebibliography}
\end{document}